\documentclass[12pt,draft]{amsart}
\usepackage[all]{xy}
\usepackage{amsfonts}
\usepackage{amssymb}
\usepackage{enumerate}
\usepackage{color}

\usepackage{stackrel}
\usepackage{pdflscape}
\usepackage{times}
\usepackage{graphics}
\usepackage{tikz-cd}
\usetikzlibrary{babel}

\newcommand{\centralquot}{N/N_c=N_1/N_c\geqslant N_2/N_c\geqslant...\geqslant N_{c-1}/N_c\geqslant N_c/N_c=1}
\newcommand{\fctgp}[2]{N_{#1}/N_{#2}}
\newcommand{\rnum}[1]{R(\phi_{#1},\psi_{#1})}
\newcommand{\rprim}[1]{R(\phi'_{#1},\psi'_{#1})}
\newcommand{\primcl}[1]{[g_{#1}N_c]_{\phi',\psi'}}
\newcommand{\cclass}[1]{[c_{#1}]_{\phi_c,\psi_c}}
\newcommand{\rclass}{[c_ig_j]_{\phi,\psi}}
\newcommand{\rcoincl}[1]{\mathfrak{R}[\phi #1,\psi #1]}

\newcommand{\commdiag}[6]{
\begin{tikzcd}[ampersand replacement=\&]
1 \arrow[r] \& #1 \arrow[r,"i"] \arrow[d, "#2"] \& #3 \arrow[r,"p"] \arrow[d, "#4"] \& #5 \arrow[r] \arrow[d, "#6"] \& 1\\
1 \arrow[r] \& #1 \arrow[r,"i"] \& #3 \arrow[r,"p"] \& #5 \arrow[r] \& 1\\
\end{tikzcd}}

\newcommand{\isol}[1]{\sqrt[N]{\gamma_{#1}(N)}}
\newcommand{\adaptedlcs}{N = \sqrt[N]{\gamma_1(N)} \geqslant \sqrt[N]{\gamma_2(N)} \geqslant ... \geqslant \sqrt[N]{\gamma_c(N)} \geqslant \sqrt[N]{\gamma_{c+1}(N)} = 1}

\newcommand{\reidclasset}[4]{\mathfrak{R(#1,#2)} = \{[#3_1 #4]_{#1,#2}, ... ,[#3_{R(#1,#2)} #4]_{#1,#2}\} }

\renewcommand{\le}{\leqslant}

\newcommand{\coht}{\operatorname{coht}}
\newtheorem{teo}{Theorem}[section]
\newtheorem{lem}[teo]{Lemma}

\theoremstyle{definition}
\newtheorem{dfn}[teo]{Definition}

\newtheorem{ex}[teo]{Example}

\def\<{\langle}
\def\>{\rangle}

\def\Z{{\mathbb Z}}

\def\End{\mathop{\rm End}\nolimits}

\def\Id{\operatorname{Id}}
\def\tr{\mathop{\rm tr}\nolimits}
\def\diag{\mathop{\rm diag}\nolimits}

\def\1{\mathbf 1}

\newcommand{\coker}{{\rm Coker\ }}

\def\notdivides{\mathrel{\kern-3pt\not\!\kern3.5pt\bigm|}}
\def\smallnotdivides{\mathrel{\kern-2pt\not\!\kern3.5pt\vert}}

\newcommand{\mbz}{\mathbb{Z}}

\newcommand{\zp}{\mathbb{Z}_p}

\begin{document}
\title[ Towards a  dichotomy  for Reidemeister zeta function  ]
{ Towards a dichotomy for the  Reidemeister zeta function }

\author[Wojciech Bondarewicz, Alexander Fel'shtyn and  Malwina Zietek]{Wojciech Bondarewicz, Alexander Fel'shtyn and  Malwina Zietek}

\address{\textsc{Wojciech Bondarewicz}\\
Instytut Matematyki\\
 Uniwersytet Szczecinski\\
ul. Wielkopolska 15, 70-451 Szczecin, Poland} 
\email{wojciech.bondarewicz@usz.edu.pl}

\address
{\textsc{Alexander Fel'shtyn}\\
Instytut Matematyki\\
 Uniwersytet Szczecinski\\
ul. Wielkopolska 15, 70-451 Szczecin, Poland} 
\email{alexander.felshtyn@usz.edu.pl}

\address{\textsc{Malwina Zietek}\\
Instytut Matematyki\\
 Uniwersytet Szczecinski\\
ul. Wielkopolska 15, 70-451 Szczecin, Poland} 
\email{malwina.zietek@gmail.com}


\thanks{The work is funded by the  Narodowe Centrum Nauki of Poland (NCN) (grant No.~\!2016/23/G/ST1/04280 (Beethoven 2)).}

\begin{abstract}
We prove a dichotomy  between rationality and a natural boundary for the analytic behavior  of the Reidemeister zeta function for automorphisms of non-finitely  generated torsion  abelian groups and for endomorphisms of groups ~$\mbz_p^d,$
where  ~$\mbz_p$~ the group of p-adic integers. As a consequence, we obtain a dichotomy for the Reidemeister zeta function of a
 continuous map of a topological space with fundamental
 group that is non-finitely generated torsion  abelian group.
We also prove  the rationality  of  the coincidence Reidemeister  zeta function  for tame endomorphisms pairs of  finitely generated torsion-free nilpotent groups, based on a weak commutativity condition.

 \end{abstract}

\maketitle
\setcounter{section}{-1}
\section{Introduction}

Let $G$ be a group and $\phi: G\rightarrow G$ an endomorphism.
Two elements $\alpha,\beta\in G$ are said to be
$\phi$-{\em conjugate} or {\em twisted conjugate,} iff there exists $g \in G$ with
$\beta=g  \alpha  \phi(g^{-1}).$
We shall write $\{x\}_\phi$ for the $\phi$-{\em conjugacy} or
{\em twisted conjugacy} class
of the element $x\in G$.
The number of $\phi$-conjugacy classes is called the {\em Reidemeister number}
of an  endomorphism $\phi$ and is  denoted by $R(\phi)$.
If $\phi$ is the identity map then the $\phi$-conjugacy classes are the usual
conjugacy classes in the group $G$.
We call the   endomorphisms $\phi $  \emph{tame} if the
Reidemeister numbers $R(\phi^n)$ are finite for all $n \in
\mathbb{N}$.
Taking a dynamical point of view, we consider the iterates of a tame endomorphism  $\phi$, and we may define following \cite{Fel91}   a Reidemeister zeta function of  $\phi$ as a power series:
\begin{align*}
R_\phi(z)&=\exp\left(\sum_{n=1}^\infty \frac{R(\phi^n)}{n}z^n\right),
\end{align*}
where $z$ denotes a complex variable.
The following problem was investigated \cite{fh}:  for which groups and endomorphisms is the Reidemeister zeta function a rational function? Is this zeta function an algebraic function?

In \cite{Fel91, fh, Li, fhw, Fel00}, the rationality of the Reidemeister zeta function $R_\phi(z)$ was proven in the following cases: the group is finitely generated and an endomorphism is eventually commutative;
the group is finite; the group is a direct sum of a finite group and a finitely generated free abelian group;
 the group is finitely generated, nilpotent and torsion-free. 
 In \cite{Wong01} the rationality of the Reidemeister  zeta function was proven
 for endomorphisms of
fundamental groups of  infra-nilmanifolds under some  sufficient conditions. Recently, the rationality of the Reidemeister zeta function was  proven for endomorphisms of fundamental groups of infra-nilmanifolds \cite{DeDu}; for endomorphisms of fundamental groups of infra-solvmanifolds of type (R) \cite{FelLee};  for automorphisms of crystallographic groups with diagonal holonomy $\mathbb{Z}_2$ and for automorphisms of almost-crystallographic groups up to dimension 3 \cite{DekTerBus};
for the right shifts of a non-finitely generated, non-abelian torsion groups $G=\oplus_{i\in\Z} F_i$,
$F_i\cong F$ and  $F$ is a finite non-abelian  group \cite{tr}.

Let $G$ be a group and $\phi, \psi : G\rightarrow G$
two endomorphisms. Two elements $\alpha,\beta\in G$ are said to be
$(\phi, \psi)-conjugate$ iff there exists $g \in G$ with
$$
\beta=\psi(g)  \alpha   \phi(g^{-1}).
$$
The number of $(\phi,\psi$)-conjugacy classes is called the Reidemeister coincidence
number of an  endomorphisms $\phi $ and $\psi $, denoted by $ R(\phi,\psi)$. If $\psi$ is the identity map then the $(\phi,id)$-conjugacy classes are the $\phi$ - conjugacy classes in the group $G$ and $R(\phi,id) = R(\phi)$.
The Reidemeister coincidence number $R(\phi,\psi)$ has useful applications in Nielsen coincidence theory.
We call the pair $(\phi,\psi)$ of endomorphisms \emph{tame} if the
Reidemeister numbers $R(\phi^n,\psi^n)$ are finite for all $n \in
\mathbb{N}$.  For such a tame pair of endomorphisms we define the    
\emph{coincidence Reidemeister zeta function}

 \begin{align*}
R_{\phi,\psi}(z)&=\exp\left(\sum_{n=1}^\infty \frac{R(\phi^n,\psi^n)}{n}z^n\right).
\end{align*}

If $\psi$ is the identity map then $R_{\phi,id}(z) = R_{\phi}(z)$.
In the theory of dynamical systems, the coincidence
  Reidemeister zeta function counts the synchronisation points of two
  maps, i.e. the points whose orbits intersect under simultaneous
  iteration of two endomorphisms; see \cite{Mi13}, for instance.

In \cite{FZ2}, in analogy to works of
Bell, Miles, Ward~\cite{BMW} and Byszewski,
  Cornelissen~\cite[\S 5]{ByCo18} about Artin-Mazur zeta function, the P\'olya--Carlson dichotomy  between rationality and a natural boundary for analytic behavior  of the coincidence Reidemeister zeta function was proven for tame pair of commuting automorphisms of non-finitely  generated torsion-free abelian groups that are subgroups of $\mathbb{Q}^d, d \geq 1$.

In \cite{FK}  P\'olya--Carlson dichotomy was proven for
coincidence Reidemeister zeta function of tame pair  of  endomorphisms of non-finitely  generated torsion-free nilpotent groups of finite Pr\"ufer rank by means of profinite completion techniques.

In this paper we present results  in support of
a dichotomy between rationality and a
natural boundary for  the Reidemeister zeta functions of endomorphisms of new broad classes of abelian groups.

We prove a dichotomy between rationality and a natural boundary for the Reidemeister zeta function of automorphisms of non-finitely generated torsion  abelian groups and of endomorphisms of the groups ~$\mbz_p^d,~ d\geq 1$, where  ~$\mbz_p$,~ $p$-prime, additive group of $p$-adic integers.

 As a consequence, we obtain a dichotomy for Reidemeister zeta function of continuous maps of topological spaces those fundamental
 group  is non-finitely generated  torsion  abelian group
 or non-finitely generated torsion-free abelian group that is subgroup of $\mathbb{Q}^d, d \geq 1$.
 
We also prove   the rationality  of  the coincidence Reidemeister  zeta function  for tame endomorphisms pairs  of  finitely generated torsion-free nilpotent groups, based on weak commutativity condition .

\bigskip

\noindent
\textbf{Acknowledgments.} 
This work was supported by the grant Beethoven 2
 of the NCN of Poland(
   grant No. 2016/23/G/ST1/04280).
   The second author is indebted to the Max-Planck-Institute for Mathematics(Bonn)  for the support and hospitality
and the possibility of the present research during his visit there.

\section{P\'olya-Carlson dichotomy for Reidemeister zeta function of automorphisms of  torsion abelian non-finitely generated groups }

Let $\phi: G \rightarrow G$  be an endomorphism of a  countable discrete abelian group  $G$.
Let  $R=\mathbb{Z}[t]$ be a polynomial  ring. Then the abelian group $G$ naturally carries the structure  of a $R$-module over the ring $R=\mathbb{Z}[t]$  where multiplication by~$t$
corresponds to application of the endomorphism:
$ tg=\phi(g)$  and extending this in a natural way to polynomials.
That is, for $g\in G$ and $ f=\sum_{n\in \mathbb{Z} }c_n t^n \in R=\mathbb{Z}[t]$ set
$$ fg= \sum_{n\in \mathbb{Z} }c_n t^ng=\sum_{n\in \mathbb{Z} }c_n \phi^n(g),$$
where all but finitely many $c_n\in \mathbb{Z}$ are zero.
This is a standard procedure for the
study of dual automorphisms of compact abelian
groups, see Schmidt~\cite{Sch} for an overview.

\newpage

\subsection{Main Theorem}

\bigskip

We will repeatedly apply the following  results
to calculate the Reidemeister numbers of iterations.
 
\begin{lem}\cite{Mi}\label{miles1}
 Let $L \subset N$ be $R$-modules and $g\in R$.\\
 Then
 \\
(1) $$ \Bigl |\frac{N}{gN}\Bigr | = \Bigl |\frac{N/L}{g(N/L)}  \Bigr |\Bigl |\frac{L}{L \cap gN} \Bigr |$$\\
 (2) If $N/L$ is finite and the map $x\to gx$ is a monomorphism of $N$ then
 $$\Bigl |\frac{N}{gN}\Bigr |=\Bigl |\frac{L}{gL}\Bigr |.$$
 \end{lem}

\begin{lem}\cite{Mi}\label{miles2}
Let $N$ be an $R$-module for which $Ass(N)$ consists 
of finitely many non-trivial principal ideals and suppose
$$ m(\mathfrak{p})= \dim_{\mathbb{K}(\mathfrak{p})}N_{\mathfrak{p}} < \infty, $$ where $\mathbb{K}(\mathfrak{p})$ denotes the  field of fractions of $R/\mathfrak{p}$ and $N_{\mathfrak{p}} = N\otimes_R\mathbb{K}(\mathfrak{p})$ is the localization of the module $N$ at $\mathfrak{p}$ .
  If $g\in R$ is such that the map $x \to gx$ is a monomorphism of $N$, then
 $N/gN$ is finite.
 \end{lem}
 If ~$\mathfrak{p}\subset R$ is a principal prime ideal,
 $\mathbb{K}(\mathfrak{p})$ above
 is a global field.
 
Additionally, if  every element of a module ~$N$ has finite
additive order then every associated
prime ideal~$\mathfrak{p}\in Ass(N)$ contains a rational prime. The coheight of a prime ideal~$\mathfrak{p}\subset R$, denoted~$\coht(\mathfrak{p})$, coincides with the Krull dimension of the domain~$R/\mathfrak{p}$ and  ~$\coht(\mathfrak{p})\leqslant 1$ for all~$\mathfrak{p}\in Ass(N)$. Our results are phrased in terms of global fields $\mathbb{K}$ hence if  every element of a $R$-module ~$N$ has finite
additive order ( this is the case when the  group $G$ above is a torsion abelian non-finitely generated group) then global field $\mathbb{K}(\mathfrak{p})$ will be 
a function field of transcedence degree one ( i.e. the Krull dimension 1) over a finite field.

Recall that a \emph{global field}~$\mathbb{K}$ of characteristic~$p>0$ is a finite extension of the rational function field~$\mathbb{F}_p(t)$, where~$t$ is an indeterminate. The \emph{places} of~$\mathbb{K}$ are the equivalence classes of absolute values on~$\mathbb{K}$, which are all non-archimedean.
For example, the \emph{infinite place} of~$\mathbb{F}_p(t)$ is given by~$|f/g|_\infty=p^{\deg(f)-\deg(g)}$ and all other places of~$\mathbb{F}_p(t)$ correspond, in the usual way, to valuation rings obtained by localizing the domain~$\mathbb{F}_p[t]$ at its non-trivial prime ideals (generated by irreducible polynomials). The places of~$\mathbb{K}$ are extensions of those just described,
and the set of all such places is denoted~$\mathcal{P}(\mathbb{K})$.
Given a finite place of $\mathbb{K}$, there
corresponds a unique discrete valuation $v$ whose precise value group is $\mathbb{Z}$.
The corresponding normalised absolute value $|\cdot|_v= |\mathcal{R}_v|^{-v(\cdot)}$,
where $\mathcal{R}_v$ is the (necessarily finite) residue class field of $v$.
For any set
of places~$S$, we
write~$|x|_S=\prod_{v\in S}|x|_v$.
Such sets of places provide a foundation for the formulas for the Reidemeister numbers of iterations of an automorphism of  a torsion abelian non-finitely generated group presented here.

The main results of this section
are the  formulas for the Reidemeister numbers and a dichotomy between rationality and a natural boundary for the analytic behaviour of the Reidemeister zeta function.
We  use  methods  of  Bell,  Miles and  Ward    to study the Artin-Mazur zeta function    of compact abelian groups automorphisms  in ~\cite[Theorem 15]{BMW},\cite[Theorem 1.1]{MW}  .

\begin{teo}\label{Main}
Let $\phi: G \rightarrow G$  be an automorphism of  a torsion abelian non-finitely generated group $G$.
Suppose that the group $G$ as  $R = \mathbb{Z}[t]$-module  satisfies the following conditions:

(1) $G$ is a Noetherian $R = \mathbb{Z}[t]$-module and its finite set of associated primes $Ass(G)$ consists entirely of non-zero
principal ideals of   the polynomial  ring $R = \mathbb{Z}[t]$,

(2) the map $g \rightarrow (t^j -1)g$  is a monomorphism of $G$ for all $j\in \mathbb{N}$\\
(equivalently, $t^j - 1 \notin \mathfrak{p}$ for all $ \mathfrak{p}\in Ass(G)$ and all $j\in
\mathbb{N}$),

(3) for each $\mathfrak{p}\in Ass(G)$, $ m(\mathfrak{p})= \dim_{\mathbb{K}(\mathfrak{p})}G_{\mathfrak{p}} < \infty$.

Then there exist  function fields ~$\mathbb{K}_1,\dots,\mathbb{K}_n$ of the form $\mathbb{K}_i=\mathbb{F}_{p(i)}(t)$, where each $p(i)$ is a rational prime, sets
of finite places~$P_i\subset \mathcal{P}(\mathbb{K}_i)$, sets of infinite places
$P_i^\infty \subset\mathcal{P}(\mathbb{K}_i)$, sets
of finite places~
 $S_i=\mathcal{P}(\mathbb{K}_i)\setminus
({P_i^\infty\cup  P_i})$, such that 
\begin{equation}\label{Reidemeister1}
R(\phi^j) = \prod_{i=1}^{n}\prod_{v\in P_i}
|t^j-1|_{v}^{-1} = \prod_{i=1}^{n}|t^j-1|_{P_i}^{-1}  =
\prod_{i=1}^{n} |t^j-1|_{P_i^\infty\cup  S_i}
\end{equation}
 for all $j\in \mathbb{N}$.
 \end{teo}

\medskip
\begin{proof}

 The  Reidemeister number of an endomorphism $\phi$ of 
 an Abelian group $G$  coincides with the cardinality of the  quotient group 
 $ \coker(\phi-\Id_G)=G/{\rm Im}(\phi-\Id_G)$
(or $\coker(\Id_G -\phi)=G/{\rm Im}(\Id_G-\phi)$).

The multiplicative set $ U = \bigcap_{\mathfrak{p}\in Ass(G)} R-{\mathfrak{p}}$
has $U\cap ann(a) = \emptyset$ for all non-zero $a\in G$, so the natural map $G\to U^{-1}G$
is a monomorphism. Identifying localizations of $R$ with subrings of $\mathbb{Q}(t)$, the
domain $ \mathfrak{R}=U^{-1}R = \bigcap_{\mathfrak{p}\in Ass(G)} R_\mathfrak{p}$ is a finite intersection
of discrete valuation rings and is therefore a principal ideal domain \cite{Matsumura}.
The assumptions (1) - (3) force $U^{-1}G $ to be a Noetherian
$ \mathfrak{R}$ - module. Hence, there is a prime filtration
$$\{0\} = G_0 \subset G_1\subset \cdot \cdot\cdot \subset G_n = U^{-1} G $$
in which $ G_i/G_{i-1} \cong \mathfrak{R}/\mathfrak{q_i}$
for non-trivial primes $\mathfrak{q_i} \subset \mathfrak{R}, 1\leq i\leq n$.
Moreover, $\mathfrak{p_i}=\mathfrak{q_i}\cap R\in Ass(G)$ for all $1\leq i\leq n$.
Identifying $G$ with its image in $U^{-1}G$ and intersecting the chain above with $G$
gives a chain 

$$\{0\} = L_0 \subset L_1\subset \cdot \cdot \cdot\subset L_n = G .$$
Considering this chain of $R$-modules, for each $1\leq i\leq n$ there is an induced inclusion
$$
 \frac{L_i}{L_{i-1}} \hookrightarrow    \frac{G_i}{G_{i-1}} \cong \frac{\mathfrak{R}}{\mathfrak{q_i}}\cong
 \mathbb{K(\mathfrak{p_i})}= K_i
$$ where field $\mathbb{K}(\mathfrak{p_i})$ is 
a function field of transcedence degree one (i.e. the Krull dimension 1) over a finite field 
and $N_i = L_i/L_{i-1}$ may be considered as a fractional ideal of $E_i = R/\mathfrak{p_i}$.
Using Lemma \ref{miles1}(1),
$$ \Bigl|\frac{L_i}{(t^j - 1)L_i}\Bigr | = \Bigl|\frac{N_i}{(t^j - 1)N_i} \Bigr |\Bigl|\frac{L_{i-1}}{L_{i-1} \cap (t^j-1)L_i} \Bigr |,$$
where $1\leq i\leq n$. Let $y\in L_i$, let $\eta$ denote the image of $y$ in $N_i$ and let $\xi_i$
denote the image of $t$ in $E_i$. If $(t^j-1)y\in L_{i-1}$ then $(\xi_i^j-1)\eta = 0$.
An assumption (2) implies $ t^j-1\notin \mathfrak{p_i}$ so $(\xi_i^j-1)\neq 0$. 
Therefore, $\eta=0$ and $y\in L_{i-1}$. It follows that $L_{i-1} \cap (t^j-1)L_i = (t^j-1)L_{i-1}$
and hence,

$$ \Bigl| \frac{L_i}{(t^j-1)L_i}\Bigr | = \Bigl|\frac{N_i}{(t^j-1)N_i}\Bigr |\Bigl|\frac{L_{i-1}}{ (t^j-1)L_{i-1}} \Bigr |,$$

Successively applying this formula to each of the modules $L_i,1\leq i\leq n$, gives,
$$ | G/(t^j-1)G| = \prod_{i=1}^{n}|N_i/(t^j-1)N_i|$$

Consider now an individual  term $|N_i/(t^j-1)N_i|$. Since char $(E_i)>0$, $E_i\cong \mathbb{F}_{p}[t]$ for some rational prime $p$ and $E_i$ is a finitely generated Dedekind domain.
 We may consider $I_i = E_i\otimes_{E_i}N_i$ as a
fractional ideal of $E_i$. Lemma \ref{miles2} and  Lemma \ref{miles1}(2) imply that
$|N_i/(\xi_i^j-1)N_i|= |I_i/(\xi_i^j-1)I_i|$ and are finite (see \cite{Mi}).
By considering $I_i/(\xi_i^j-1)I_i$ as a $E_i$-module, finding a composition series for this module and    successively localizing at each of its associated primes to obtain multiplicities, it follows that
$$|I_i/(\xi_i^j-1)I_i| = \prod_{\mathfrak{m}\in Ass(I_i/(\xi_i^j-1)I_i)}q_{\mathfrak{m}}^{\delta_\mathfrak{m}(\xi_i,I_i)},$$ where $q_{\mathfrak{m}}= |E_i/\mathfrak{m}|$ and $\delta_\mathfrak{m}(\xi_i,I_i) = \dim_{E_i/\mathfrak{m}}(I_i/(\xi_i^j-1)I_i)_\mathfrak{m}$. Let 
$$P_i=\{\mathfrak{m}\in Spec(E_i) : I_{\mathfrak{m}} \neq K_i\}.$$
 It follows that the product above may be taken over all $\mathfrak{m}\in P_i$ to yield the same result. Each localization $(E_i)_{\mathfrak{m}}$ is a distinct valuation ring of $K_i$ and $P_i$ may be identified with a set of finite places of the global field $K_i$.
Hence, since $\delta_\mathfrak{m}(\xi_i,E_i) = v_\mathfrak{m}(\xi_i^j-1)$, finally we have  
$$ |I_i/(\xi_i^j-1)I_i| = \prod_{\mathfrak{m}\in P_i}q_{\mathfrak{m}}^{\delta_\mathfrak{m}(\xi_i,E_i)}= \prod_{\mathfrak{m}\in P_i}q_{\mathfrak{m}}^{v_\mathfrak{m}(\xi_i^j-1)} = \prod_{\mathfrak{m}\in P_i}|\xi_i^j-1 |_{\mathfrak{m}}^{-1},$$
where $|\cdot|_{\mathfrak{m}}$ is the normalised absolute value arising from $E_{\mathfrak{m}}$. 
This concludes the proof
of the formula $R(\phi^j) = \prod_{i=1}^{n}\prod_{v\in P_i}
|\xi_i^j-1|_{v}^{-1} = \prod_{i=1}^{n}|\xi_i^j-1|_{P_i}^{-1} $.

 Applying the Artin product formula \cite{Weil}
  gives
\begin{equation}\label{main_formula2}
R(\phi^j)= \prod_{i=1}^{n}|\xi_i^j-1|_{P_i}^{-1}=
\prod_{i=1}^{n}
|\xi_i^j-1|_{P_i^\infty\cup S_i}.
\end{equation}

\end{proof}

\medskip
We remind the definition of a natural boundary.
\begin{dfn} 
\rm
Suppose that an analytic function $F$ is defined somehow in a region $D$
of the complex plane.
If there is no point of the boundary $\partial D$ of $D$ over which $F$ can be analytically continued, then 
$\partial D$ is called a \emph{natural boundary} for $F$  . 
\end{dfn}

\begin{teo}\label{dichotomy}
Let $\phi:G\rightarrow G$ be an automorphism of a torsion abelian non-finitely generated group $G$ and define a number $h(\phi)=\sum_{i=1}^{r}\log p(i)$,  where the rational primes $p(i)$ are the same as in   Theorem \ref{Main}. Suppose that the group $G$ as an $R=\mathbb{Z}[t]$-module satisfies the conditions (1)--(3) in   Theorem \ref{Main}. If the Reidemeister zeta function $R_\phi(z)$ has a radius of convergence $R=e^{-h(\phi)}$, then it is either a rational function, namely
$$R_\phi(z)=(1-e^{h(\phi)}z)^{-1}$$
or the circle $|z|=e^{-h(\phi)}$ is a natural boundary for
the function  $R_\phi(z)$.
\end{teo}

\begin{proof}
By Theorem \ref{Main}, there exist function fields $\mathbb{K}_1,...,\mathbb{K}_n$ of the form $\mathbb{K}_i=\mathbb{F}_{p(i)}(t)$, where each $p(i)$ is a rational prime and sets of finite places $P_i\subset \mathcal{P}(\mathbb{K}_i)$ for $i=1,...,n$, such that:
$$R(\phi^j)=\prod_{i=1}^{n}\prod_{v\in P_i}|t^j-1|_v^{-1}$$
for all $j\in\mathbb{N}$. If we denote for each ~$i=1,...,n $ by
$P_i^\infty \subset\mathcal{P}(\mathbb{K}_i)$ the infinite place and  by
 $S_i=\mathcal{P}(\mathbb{K}_i)\setminus
({P_i^\infty\cup  P_i})$ set
of remaining finite places then the above formula can be also expressed as
$$R(\phi^j)=\prod_{i=1}^{n}\prod_{v\in P_i^\infty\cup  S_i}|t^j-1|_v.$$ Since
$$\prod_{i=1}^{n}|t^j-1|_{P_i^\infty} = \prod_{i=1}^{n}p(i)^{deg(t^j-1)}  = \prod_{i=1}^{n} p(i)^j = e^{(\sum_{i=1}^{n} \log p(i))j} = e^{h(\phi)j},$$
then
$$R(\phi^j)=\prod_{i=1}^{n}\prod_{v\in P_i^\infty\cup  S_i}|t^j-1|_v = e^{h(\phi)j}\prod_{i=1}^{n}\prod_{v\in S_i}|t^j-1|_v.$$
Set $a_j = \prod_{v\in S_i}|t^j-1|_v$, define $F(z)=\sum_{j=1}^{\infty}a_jz^j$ and note that 
$$F(e^h(\phi)z) = zR'_\phi(z)/R_\phi(z).$$
By hypothesis $R_\phi(z)$ has a radius of convergence $e^{h(\phi)}$, so it follows that the radius of convergence of $F(z)$ is 1. 
If $S_i=\{(t)\}$ for all $i=1,..,n$, then $a_j=1$ and it follows immediately that $R_\phi(z) = (1-e^{h(\phi)}z)^{-1}$.\\
If not, there exists $1\leqslant m\leqslant n$ such that $S_m$ contains a place $w$ corresponding to a polynomial of degree $d_w\geqslant 1$ such that $w\neq t$, so that means that $w\nmid t$, equivalently $ord_w(t) = 0$ and thus $|t|_w=1$. Set $p=p(i)$ and $j_k=l_wp^k$, where $l_w$ denotes the multiplicative order of the image of $t$ in the finite residue field at $w$, i.e. $t^{l_w}=1$ . 
From the binomial expansion and the fact that $p=p(i)=$char$(\mathbb{K}_i)>0$ we obtain
$$t^{l_wp^k}-1=(t^{l_w}-1)^{p^k}=(w\cdot f(t))^{p^k}=w^{p^k}\cdot f(t)^{p^k}.$$
Since $w\nmid f(t)$, then $ord_w(t^{l_wp^k}-1)=p^k$. 
Let us denote the degree of $w$ as $d_w$.
Then
$$|t^{j_k}-1|_w=p^{-ord_w(t^{j_k}-1)\cdot deg w}=p^{-ord_w(t^{l_wp^k}-1)\cdot d_w}=p^{-p^k\cdot d_w}<1.$$
For all $1\leqslant j$ we have that $a_j\leqslant |t^j-1|_w$, hence
$$\limsup_{j_k\rightarrow\infty}a_{j_k}^{1/j_k}\leqslant p^{-d_w/l_w}.$$
Therefore $F$ contains a sequence of partial sums that is uniformly convergent for $|z|<p^{d_w/l_w}$. Hence, the series  $F$ is overconvergent, so series may be  written as a sum of a series convergent on $z<|p|$ and a lacunary series and hence the unit circle is a natural boundary for $F$ ( see \cite{Segal}, sec. 6.2). It follows that  the circle $|z|=e^{-h(\phi)}$ is a natural boundary for $R_\phi(z)$.
\end{proof}

\subsection{Examples}
To give an example of the dichotomy for analytic behavior of the Reidemeister zeta function we use  the  calculations  of Miles and Ward  in \cite{MW} for the number of periodic points of the dual compact
abelian group endomorphism $\widehat{\phi}$ .
 Let us consider a ring $M = \mathbb{F}_p[t^{\pm 1},(t-1)^{\pm 1}]$, and  an endomorphism $\phi: g \rightarrow tg$ which is the multiplication by $t$ on $M$. It follows from \cite{MW}, Example 4.1 that $M$ has a structure of a $\mathbb{Z}[t]$-module.
 Then using the  calculations  of Miles and Ward  in \cite{MW}, Example 3.1 and  the formula  \ref{Reidemeister1} for the Reidemeister numbers
$$R(\phi^j) = \prod_{i=1}^{n} |t^j-1|_{P_i^\infty\cup  S_i},$$ we obtain
$$R(\phi^j) = |t^j - 1|_\infty |t^j - 1|_{(t)} |t^j - 1|_{(t-1)}.$$
Since
$|t^j - 1|_\infty = p^j,$
$|t^j - 1|_{(t)} = p^{-ord_{t}(t^j-1)\cdot \deg t} = 1,$
\begin{equation*}
\begin{split}
|t^j - 1|_{(t-1)}  = p^{-ord_{(t-1)}(t^j-1)\cdot \deg (t-1)}
 = p^{-p^{ord_p(j)}}= p^{-|j|_p^{-1}},\\
\end{split}
\end{equation*}
then finally
$R(\phi^j) = p^{j-\upsilon(j)},$
where $\upsilon(j) = |j|_p^{-1}$.

By the Cauchy -- Hadamard
formula, the radius of convergence $R$ of the Reidemeister zeta function $R_\phi(z)$ is equal to
$$R = \left(\limsup_{j\rightarrow\infty}\sqrt[j]{\frac{p^j-\upsilon(j)}{j}}\right)^{-1} = p^{-1}.$$ 
Let us define $F(z) = \sum_{j=1}^\infty a_jz^j$, where $a_j = p^{-\upsilon(j)}$. Then the radius $R_F$ of convergence of $F$ is equal to
$$R_F = \left(\limsup_{j\rightarrow\infty}\sqrt[j]{p^{-\upsilon(j)}}\right)^{-1} = 1.$$
Moreover, $$F(pz) = zR'_\phi(z)/R_\phi(z),$$
so if the Reidemeister zeta function $R_\phi(z)$ has the analytic continuation beyond the circle $|z|=p^{-1}$, then $F$ has the analytic continutation beyond the unit circle. However, $F$ contains a sequence of partial sums that is uniformly convergent for $|z|<p$, since for $j_k=p^k$, we have $\limsup_{j_k\rightarrow\infty}a_{j_k}^{1/j_k} = p^{-1}$. It follows that the series $F$ is overconvergent, so it may be written as a sum of a series convergent on $|z|<p$  and a lacunary series, and hence the unit circle is a natural boundary for $F$. Therefore $|z|=p^{-1}$ is a natural boundary for $R_\phi(z)$.

\section{ P\'olya-Carlson dichotomy for the Reidemeister  zeta function of  endomorphisms of the groups ~ $\mbz_p^d$ }

In this section we prove
a P{\'o}lya--Carlson dichotomy between rationality and a
natural boundary for the analytic behaviour of  the Reidemeister  zeta function  for     endomorphisms of groups ~$\mbz_p^d,~ d\geq 1$, where  ~$\mbz_p$,~ p-prime,  denotes additive group of p-adic integers. The group ~$\mbz_p$ is the most basic infinite pro-p group, it is totally disconnected, compact, abelian, torsion-free group.The field
of $p$-adic numbers is denoted by $\mathbb{Q}_p$
and the $p$-adic absolute value (as well
as its unique extension to the algebraic closure
$\overline{\mathbb{Q}}_p$) by $\lvert \cdot \rvert_p$.

We need the following statement

\begin{lem}(cf. \cite{BMW})\label{generating}
Let~$Z(z)=\sum_{n= 1}^{\infty}R(\phi^n)z^n$.
If~$R_{\phi}(z)$ is rational then~$Z(z)$ is rational.
If~$R_{\phi}(z)$ has an analytic continuation beyond
its circle of convergence, then so does~$Z(z)$ too.
In particular, the existence of a natural boundary
at the circle of convergence for~$Z(z)$ implies the
existence of a natural boundary for~$R_{\phi}(z)$.
\end{lem}

\begin{proof}
This follows from the fact that~$Z(z)= z\cdot R_{\phi}(z)^{'}/R_{\phi}(z)$.
\end{proof}

One of the important links between the arithmetic
properties of the coefficients of a complex power
series and its analytic behaviour is given
by the P{\'o}lya--Carlson theorem  \cite{Segal}.

\medskip{\noindent \bf P\'{o}lya--Carlson Theorem.}
{\it A power series with integer coefficients
and radius of convergence~$1$ is either rational or has the
unit circle as a natural boundary.}\medskip

	\begin{lem}\label{p}
		$\End(\mbz_p)=\mbz_p$ for abelian group $\mbz_p$. 
	\end{lem}
	\begin{proof} Let $\phi\in\End(\mbz_p)$. We have $p^n\phi(x)=\phi(p^nx)$. Then $\phi(p^n\mbz_p)\subset p^n\mbz_p$, so $\phi$ is continuous. For every $x\in\mbz_p$ there exists a sequence of integers $x_n$ converging to $x$. Then $$\phi(x)=\lim \phi(x_n)=\lim x_n\phi(1)=\phi(1)x,$$ so $\phi$ is a multiplication by $\phi(1)$. 
	\end{proof}
	%\begin{Wni}
	%	$\End(\mbz_p^d)=M_d(\mbz_p)$.
	%\end{Wni}
	
	Let $\phi\in\End(\mbz_p)$, then $\phi(x)=ax$, where $a\in\mbz_p$. We have $\phi^n(x)=a^nx$. By definition, $$y\sim_{\phi}x\Leftrightarrow\exists b\in\mbz_p: y=b+x-ab=x+b(1-a)\Leftrightarrow y\equiv x(\mathrm{mod}(1-a)).$$
	
	This implies that $R(\phi)=|\mbz_p/(1-a)\mbz_p|$. But $$(1-a)\zp=p^{v_p(1-a)}\zp=|1-a|_p^{-1}\mbz_p,$$ so we can write $R(\phi)=|1-a|_p^{-1}=|a-1|_p^{-1}$ and, more generally,
	
	 $R(\phi^n)=|1-a^n|_p^{-1}=|a^n -1|_p^{-1},$ for all ~ $n \in \mathbb{N}.$

	Now consider a group $\mbz_p^d, ~d\geq 2$. It follows easily from Lemma \ref{p}, that $\End(\zp^d)=\mathrm{M}_d(\zp)$. For any matrix $A\in M_d(\zp)$ there exists a diagonal matrix $D\in M_d(\zp)$ and unimodular matrices  $E,F\in M_d(\zp)$ such that $D=EAF$. 
	
	%%%%%%%%%%%%%%%%%%%%%%%%%
	\begin{lem}\label{coker}
	For endomorphism $\phi_p:\zp^d\rightarrow \zp^d$ we have
	$$R(\phi_p)=\#\mathrm{Coker}(1-\phi_p)=|\det(\Phi_p - \mathrm{Id})|_p^{-1},$$ where $\Phi_p$ is a matrix of $\phi_p$.
	\end{lem}
	
	\begin{proof}
		 Let matrices $D,E,F\in M_d(\zp)$ be such that $D=E(\mathrm{Id}-\Phi_p)F$, where $D=(a_i)$ is diagonal matrix, $a_i\in\zp, 1\leq i \leq d$, and  matrices $E,F$ are unimodular. Then we have

$$
R(\phi_p)=\#\mathrm{Coker}(1-\phi_p)=|\zp^d:(\mathrm{Id}-\Phi_p)\zp^d|=|\mbz_p^{d}:D\mbz_p^d|=\\
$$
$$
=| \zp:a_1\zp|\cdot|\zp:a_2\zp|\cdot...\cdot|\zp:a_d\zp|=
|a_1|_p^{-1}\cdot...\cdot|a_d|_p^{-1}=\\
$$
$$
=|\det(D)|_p^{-1}=|\det(\mathrm{Id}-\Phi_p)|_p^{-1}=|\det(\Phi_p - \mathrm{Id})|_p^{-1}. 
$$

\end{proof}

In order to handle the sequence $R(\phi^n)=|a^n -1|_p^{-1}, ~ n \in \mathbb{N}$ more easily, we need a way to evaluate expressions of the form $|a^n -1|_p$ when $|a|_p=1$. The following technical lemma is useful.

\begin{lem} (\cite [Lemma 4.9]{Mi07})\label{key}
Let $K_ v $ be a non-archimedean local field and suppose $ x \in
K_v $ has $|x|_v = 1$ and infinite multiplicative order.
 Let $p > 0$ be the characteristic of the residue field $F_ v$ and $\gamma \in \mathbb{N}$ the multiplicative order
of the image of $x$ in $F_ v$ . Then $|x^n - 1|_v = 1$ whenever 
$(\gamma, n) = 1$ and
$\gamma \neq 1$. Furthermore, there are constants $ 0 < C < 1 $ and $r_ 0\geq 0$ such that whenever $n = k\gamma p^r$ with $(p, k) = 1$ and $r > r_ 0 $, then
$
|x^n- 1|_v = C |p|_v^r  \,\,\,  \textrm{ if }  \,\, char(K_ v )= 0 .
$

\end{lem}

Now  we prove
a P{\'o}lya--Carlson dichotomy between rationality and a
natural boundary for the analytic behaviour of  the Reidemeister  zeta function  for     endomorphisms of groups ~$\mbz_p^d,~ d\geq 1$.

\begin{teo}
	Let   $\phi_p:\zp^d\rightarrow \zp^d$ be a tame endomorphism 
	and $ \lambda_1,..,\lambda_d \in \overline{\mathbb{Q}}_p $ be the eigenvalues of $\Phi_p$, including multiplicities.
	Then the Reidemeister zeta function $R_{\phi_p}(z)$ is either a
	rational function or it has the unit circle as a natural boundary.
	Furthermore, the
  latter occurs if and only if
  $\lvert \lambda_i \rvert_p= 1$ for some
  $i \in \{1,\ldots,d\}$.
	\end{teo}

\begin{proof}
Firstly, we consider separately the case of the group $\zp$ as it illustrates some of ideas needed for the proof of the dichotomy in general case when
$ d\geq 1$.  Lemma \ref{p} yields  ~$\phi_p(x)=ax$, where $a\in\mbz_p$.
Hence $ |a|_p \le 1.$ Then the Reidemeister numbers
$R(\phi_p^n)=|a^n -1|_p^{-1},$ for all  ~ $n \in \mathbb{N}.$ 
 If  $ |a|_p < 1$, then ~$R(\phi_p^n)=|a^n -1|_p^{-1} =1,$ for all ~ $n \in \mathbb{N}$. Hence the radius of convergence of $R_{\phi_p}(z)$ equals 1  and the Reidemeister zeta function $ R_{\phi_p}(z)=\frac{1}{1 - z}$ is
a rational function. 

From now on, we shall write $ a(n) < < b(n)$ if there is a constant $c$ independent of $n$  for which $ a(n) < c\cdot b(n)$.
When $ |a|_p = 1$, we show that the radius of convergence of $R_{\phi_p}(z)$ equals 1
by deriving the bound 
\begin{equation}\label{bound}
\frac{1}{n} < < |a^n -1|_p \le 1
\end{equation}

Upper bound  in (\ref{bound}) follows from the definition of the p-adic norm. We may suppose that $|a^n -1|_p < 1$. Let $F$ denote the smallest field which contains $\mathbb{Q}_p$ and is both algebraically closed and complete with respect to $|\cdot|_p  $ .
The p-adic logarithm $\log_p$ is defined as
$$
\log_p(1 + z) =\sum_{n=1}^\infty \frac{(-1)^{n+1}}{n}z^n, 
$$
and converges for all $z\in F$ such that $ |z|_p < 1.$
Setting $z= a^n -1$ we get
$$
\log_p(a^n) = (a^n -1) - \frac{(a^n -1)^2}{2} + \frac{(a^n -1)^3}{3} - . . .
$$
and so $|\log_p(a^n)|_p \le |a^n -1|_p$. We always have 
$$
\frac{1}{n} < < |n\log_p(a)|_p = |\log_p(a^n)|_p ,
$$
so this establishes (\ref{bound}).

The bound (\ref{bound}) implies  by Cauchy - Hadamard
formula that the radius of convergence of $R_{\phi_p}(z)$ equals 1.
Hence  it remains to show that the Reidemeister zeta function $R_{\phi_p}(z)$ is irrational if $ |a|_p = 1$.
 Then $R_{\phi_p}(z)$ has the unit circle as a natural boundary by the  Lemma \ref{generating}  and by the P{\'o}lya--Carlson Theorem.
For a contradiction, assume that Reidemeister zeta function $R_{\phi_p}(z)$  is rational.
Then Lemma \ref{generating} implies that the function
~$Z_p(z)=\sum_{n= 1}^{\infty}R(\phi_p^n)z^n$ is rational also.
Hence the sequence~$R(\phi_p^n)$ satisfies a linear recurrence relation.

Let~$q \neq p $ be a rational prime, and define
\[
n(e)=q^e \gamma p^{r},
\]
where integer constant~$r\geqslant 0$ 
and~$e\geqslant 1$. Applying
Lemma~\ref{key}, we see
that
\[
R(\phi_p^{kn(e)})= R(\phi_p^{n(e)})
\]
whenever~$k$ is coprime to~$n(e)$. Hence the sequence~~$R(\phi_p^n)$
assumes infinitely many values infinitely often, and so it
cannot satisfy a linear recurrence
by a result of Myerson and van der
Poorten~\cite[Prop.~2]{MyPoo}, giving a contradiction.

Now we consider the general case of a tame endomorphism
$\phi_p:\zp^d\rightarrow \zp^d,~ d\geq 1$.
According to the Lemma \ref{coker} we have
	$$R(\phi_p^n)=\#\mathrm{Coker}(1-\phi_p^n)=|\det(\Phi_p^n - \mathrm{Id})|_p^{-1} = \prod_{i=1}^{d}|\lambda_i^n-1|_{p}^{-1} ,$$
	 where $\Phi_p$ is a matrix of $\phi_p$
and $\lambda_1,\lambda_2,...\lambda_d \in \overline{\mathbb{Q}}_p $ are the eigenvalues of $\Phi_p$, including multiplicities.
 The polynomial
    $\prod_{i=1}^{d} (X - \lambda_i )$
     has coefficients in
    $\mathbb{Z}_p$; in particular,
    $\lvert \lambda_i \rvert_p \le 1$ for
    every $i \in \{1,\ldots,d\}$ (see \cite{FK}).
    If  $ |\lambda_i|_p < 1$, for every $i \in \{1,\ldots,d\}$    then ~$R(\phi_p^n)= \prod_{i=1}^{d}|\lambda_i^n-1|_{p}^{-1}  =1,$ for all ~ $n \in \mathbb{N}$. Hence the radius of convergence of $R_{\phi_p}(z)$ equals 1  and the Reidemeister zeta function $ R_{\phi_p}(z)=\frac{1}{1 - z}$ is
a rational function. 
If  $ |\lambda_i|_p = 1$, for some  $i \in \{1,\ldots,d\}$ ~ 
    then  the bound (\ref{bound}) implies the bound
    \begin{equation}\label{bound1}
\frac{1}{n^d} < < R(\phi_p^n)= \prod_{i=1}^{d}|\lambda_i^n-1|_{p}^{-1} \le 1 .
\end{equation}
Hence the radius of convergence of the Reidemeister zeta function  $R_{\phi_p}(z)$ equals 1 
by the Cauchy--Hadamard formula and the bound (\ref{bound1}).
Now for the proof of the theorem it remains to show that 
the Reidemeister zeta  function $R_{\phi_p}(z)$ is irrational if
$ |\lambda_i|_p = 1 $ for some $i \in \{1,\ldots,d\}.$
Then $R_{\phi_p}(z)$ has the unit circle as a natural boundary by the  Lemma \ref{generating}  and by the P{\'o}lya--Carlson Theorem.
For a contradiction, assume that Reidemeister zeta function $R_{\phi_p}(z)$  is rational.
Then Lemma \ref{generating} implies that the function
~$Z_p(z)=\sum_{n= 1}^{\infty}R(\phi_p^n)z^n$ is rational also.
Hence the sequence~$R(\phi_p^n)$ satifies a linear recurrence relation.
Let~$q \neq p $ be a rational prime, and define
$
n(e)=q^e \gamma p^{r},
$
where integer constant~$r\geqslant 0$ 
and~$e\geqslant 1$. Applying
Lemma~\ref{key}, we see
that
$
R(\phi_p^{kn(e)})= R(\phi_p^{n(e)})
$
whenever~$k$ is coprime to~$n(e)$. Hence the sequence~~$R(\phi_p^n)$
assumes infinitely many values infinitely often, and so it
cannot satisfy a linear recurrence
by a result of Myerson and van der
Poorten~\cite[Prop.~2]{MyPoo}, giving a contradiction.

\end{proof}

\section{ P\'olya-Carlson dichotomy for the Reidemeister zeta function of continuous map}

We assume  $X$ to be a path- connected
topological space admitting a universal cover 
$p:\tilde{X}\rightarrow X$ and $f:X\rightarrow X$ to be a continuous map.
Let  $\tilde{f}:\tilde{X}\rightarrow \tilde{X}$ a lifting of $f$, ie. $p\circ\tilde{f}=f\circ p$.
Two liftings $\tilde{f}$ and $\tilde{f}^\prime$ are called
{\sl conjugate} if there is a $\gamma\in\Gamma\cong\pi_1(X)$
such that $\tilde{f}^\prime = \gamma\circ\tilde{f}\circ\gamma^{-1}$.
The subset $p(Fix(\tilde{f}))\subset Fix(f)$ is called
{\sl the fixed point class of $f$ determined by the lifting class $[\tilde{f}]$}.
The number of lifting classes of $f$ (and hence the number
of fixed point classes, empty or not) is called the {\sl Reidemeister Number} of $f$,
denoted $R(f)$.
This is a positive integer or infinity.
$R(f)$  is homotopy invariant.

Let a
specific lifting $\tilde{f}:\tilde{X}\rightarrow\tilde{X}$ be chosen
as reference.
Let $\Gamma$ be the group of
covering translations of $\tilde{X}$ over $X$.
Then every lifting of $f$ can be written uniquely
as $\gamma\circ \tilde{f}$, with $\gamma\in\Gamma$.
So elements of $\Gamma$ serve as coordinates of
liftings with respect to the reference $\tilde{f}$.
Now for every $\gamma\in\Gamma$ the composition $\tilde{f}\circ\gamma$
is a lifting of $f$ so there is a unique $\gamma^\prime\in\Gamma$
such that $\gamma^\prime\circ\tilde{f}=\tilde{f}\circ\gamma$.
This correspondence $\gamma\rightarrow\gamma^\prime$ is determined by
 the reference $\tilde{f}$, and is obviously a homomorphism.

\begin{dfn}
The endomorphism $\tilde{f}_*:\Gamma\rightarrow\Gamma$ determined
by the lifting $\tilde{f}$ of $f$ is defined by
$$
  \tilde{f}_*(\gamma)\circ\tilde{f} = \tilde{f}\circ\gamma.
$$
\end{dfn}
 
It is well known that $\Gamma\cong\pi_1(X)$.
We shall identify $\pi=\pi_1(X,x_0)$ and $\Gamma$ in the following way.
Pick base points $x_0\in X$ and $\tilde{x}_0\in p^{-1}(x_0)\subset \tilde{X}$
once and for all.
Now points of $\tilde{X}$ are in 1-1 correspondence with homotopy classes of paths
in $X$ which start at $x_0$:
for $\tilde{x}\in\tilde{X}$ take any path in $\tilde{X}$ from $\tilde{x}_0$ to 
$\tilde{x}$ and project it onto $X$;
conversely for a path $c$ starting at $x_0$, lift it to a path in $\tilde{X}$
which starts at $\tilde{x}_0$, and then take its endpoint.
 In this way, we identify a point of $\tilde{X}$ with
a path class $<c>$ in $X$ starting from $x_0$. Under this identification,
$\tilde{x}_0=<e>$ is the unit element in $\pi_1(X,x_0)$.
The action of the loop class $\alpha = <a>\in\pi_1(X,x_0)$ on $\tilde{X}$
is then given by
$$
 \alpha = <a> : <c>\rightarrow \alpha . c = <a.c>.
$$
Now we have the following relationship between $\tilde{f}_*:\pi\rightarrow\pi$
and
$$
 f_*  :  \pi_1(X,x_0) \longrightarrow \pi_1(X,f(x_0)).
$$

\begin{lem}
Suppose $\tilde{f}(\tilde{x}_0) = <w>$.
Then the following diagram commutes:
$$
\begin{array}{ccc}
  \pi_1(X,x_0)  &  \stackrel{f_*}{\longrightarrow}  &  \pi_1(X,f(x_0))  \\
                &  \tilde{f}_* \searrow \;\; &  \downarrow w_*   \\
                 &                           &  \pi_1(X,x_0)
\end{array}
$$
\end{lem}

\begin{lem}\cite{j}\label{Boju}
Lifting classes of $f$ are in 1-1 correspondence with
$\tilde{f}_*$-conjugacy classes in $\pi$,
the lifting class $[\gamma\circ\tilde{f}]$ corresponding
to the $\tilde{f}_*$-conjugacy class of $\gamma$.
We therefore have $R(f) = R(\tilde{f}_*)$.
\end{lem}

Taking a dynamical point of view,
 we consider the iterates of $f$ assume that $R(f^n)<\infty$ for all $n>0$.
 The Reidemeister zeta function of $f$ 
  is defined \cite{Fel91}
 as power series:
 \begin{eqnarray*}
 R_f(z)
 & := &
 \exp\left(\sum_{n=1}^\infty \frac{R(f^n)}{n} z^n \right),
 \end{eqnarray*}
 
Using Lemma \ref{Boju} we may apply the theorems of 
previous section about  the Reidemeister zeta function for automorphisms
 to the Reidemeister zeta function of continuous maps.
 
\begin{teo}\label{Main3}
Suppose that the fundamental group $ \Gamma $ of the topological space $X$ is    a torsion abelian non-finitely generated group and $\tilde{f}_*: \Gamma \rightarrow \Gamma$ is an automorphism .
Suppose that the group $\Gamma$ as  $R = \mathbb{Z}[t]$- module  satisfies the following conditions:

(1) the set of associated primes $Ass(\Gamma)$ is finite and consists entirely of non-zero
principal ideals of   the polynomial  ring $R = \mathbb{Z}[t]$,

(2) the map $g \rightarrow (t^j -1)g$  is a monomorphism of $\Gamma$ for all $j\in \mathbb{N}$\\
(equivalently, $t^j - 1 \notin \mathfrak{p}$ for all $ \mathfrak{p}\in Ass(\Gamma)$ and all $j\in
\mathbb{N}$),

(3) for each $\mathfrak{p}\in Ass(\Gamma)$, $ m(\mathfrak{p})= \dim_{\mathbb{K}(\mathfrak{p})}G_{\mathfrak{p}} < \infty$.

Then there exist
function fields ~$\mathbb{K}_1,\dots,\mathbb{K}_n$ of the form $\mathbb{K}_i=\mathbb{F}_{p(i)}(t)$, where each $p(i)$ is a rational prime, sets
of finite places~$P_i\subset \mathcal{P}(\mathbb{K}_i)$,
  sets of infinite places
$P_i^\infty \subset\mathcal{P}(\mathbb{K}_i)$, sets
of finite places~
 $S_i=\mathcal{P}(\mathbb{K}_i)\setminus
({P_i^\infty\cup  P_i})$ such that 
\begin{equation}\label{Reidemeister2}
R(f^j) = R(\tilde{f}_*^j)= \prod_{i=1}^{n}\prod_{v\in P_i}
|t^j-1|_{v}^{-1} = \prod_{i=1}^{n}|t^j-1|_{P_i}^{-1}  =
\prod_{i=1}^{n} |t^j-1|_{P_i^\infty\cup  S_i}
\end{equation}
 for all $j\in \mathbb{N}$.
 If the Reidemeister zeta function $R_f(z)=R_{\tilde{f}_*}(z)$ has radius of convergence $R=e^{-h(f)}$, where $h(f)=\sum_{i=1}^{r}\log p(i)$,  then it is either a rational function, namely
$$R_f(z)= R_{\tilde{f}_*}(z)=(1-e^{h(f)}z)^{-1}$$
or the circle $|z|=e^{-h(f)}$ is a natural boundary for the function $R_f(z)=R_{\tilde{f}_*}(z)$.
 
\end{teo}

Similarly, theorem 5.8 in \cite{FeZi20} implies the following   P\'olya-Carlson dichotomy between rationality and a natural boundary
 for  the Reidemeister zeta function of continuous map.
 
\begin{teo}\label{Main2}
Suppose that the fundamental group $ \Gamma $ of the topological space $X$ is  a  countable  abelian group that is a subgroup  of~$\mathbb{Q}^d$, where $d\geqslant 1$ and $\tilde{f}_*: \Gamma \rightarrow \Gamma$ is an automorphism.
Suppose that the group $\Gamma$ as  $R = \mathbb{Z}[t]$- module  satisfies the following conditions:

(1) the set of associated primes $Ass(\Gamma)$ is finite and consists entirely of non-zero
principal ideals of   the polynomial  ring $R = \mathbb{Z}[t]$,

(2) the map $g \rightarrow (t^j -1)g$  is a monomorphism of $\Gamma$ for all $j\in \mathbb{N}$\\
(equivalently, $t^j - 1 \notin \mathfrak{p}$ for all $ \mathfrak{p}\in Ass(\Gamma)$ and all $j\in
\mathbb{N}$),

(3) for each $\mathfrak{p}\in Ass(\Gamma)$, $ m(\mathfrak{p})= \dim_{\mathbb{K}(\mathfrak{p})}G_{\mathfrak{p}} < \infty$.

Then there exist
algebraic number fields~ $\mathbb{K}_1,\dots,\mathbb{K}_n$, 

sets
of finite places~ $P_i \subset \mathcal{P}(\mathbb{K}_i)$,
 $S_i=\mathcal{P}(\mathbb{K}_i)\setminus
P_i$,
 and
elements~$\xi_i\in \mathbb{K}_i$, no one of which is a root of
unity for~$i=1,\dots,n$, such that 
\begin{equation}\label{Reidemeister3}
R(f^j) =\prod_{i=1}^{n}\prod_{v\in P_i}
|\xi_i^j-1|_{v}^{-1} = \prod_{i=1}^{n}|\xi_i^j-1|_{P_i}^{-1}  =
\prod_{i=1}^{n}
|\xi_i^j-1|_{P_i^\infty\cup S_i}
\end{equation}
 for all $j\in \mathbb{N}$.
 
 Suppose that the last
product in~\eqref{Reidemeister3} only involves
finitely many places and that~$|\xi_i|_v\neq 1$ for all~$v$
in the set of infinite places $P_i^\infty$ of $\mathbb{K}_i$ and all $i=1,\dots,n$. \\
Then the Reidemeister zeta function $R_{f}(z)=R_{\tilde{f}_*}(z)$ is
either rational function or has a natural boundary at its circle of
convergence, and the latter occurs if and only if~$|\xi_i|_v=1$
for some~$v\in S_i$,~$1\le i\le n$.
\end{teo}

\section{The rationality  of  the coincidence Reidemeister  zeta function  for endomorphisms of  finitely generated torsion-free nilpotent groups}

\begin{ex}(\cite{FK}, Example 1.3)\label{ex1}
  Let $G = \mathbb{Z}$ be the infinite cyclic group, written
  additively, and let
  $$
  \phi \colon \mathbb{Z} \to \mathbb{Z}, \quad x \mapsto d_\varphi x \qquad \text{and} \qquad
  \psi \colon \mathbb{Z} \to \mathbb{Z}, \quad x \mapsto d_\psi x
  $$
  for $d_\phi, d_\psi \in \mathbb{Z}$.  The coincidence Reidemeister number $R(\phi, \psi)$ of  endomorphisms $\phi,\psi $ of  an Abelian group $G$  coincides with the cardinality of the  quotient group $ \coker(\phi-\psi)=G/{\rm Im}(\phi-\psi)$
(or $\coker(\psi -\phi)=G/{\rm Im}(\psi-\phi)$). 
Hence we have 
  $$
  R(\phi^{\, n},\psi^{\, n}) =
  \begin{cases}
    \lvert d_\psi^{\, n} - d_\phi^{\, n} \rvert & \text{if $d_\phi^{\, n}
      \not = d_\psi^{\, n}$,} \\
    \infty & \text{otherwise}.
  \end{cases}
  $$
  Consequently, $(\phi,\psi)$ is tame precisely when
  $\lvert d_\phi \rvert \not = \lvert d_\psi \rvert$ and, in this case, 
  $$
  R_{\phi,\psi}(z) = \frac{1 - d_2 z}{1 - d_1 z} \qquad \text{where
    $d_1 = \max \{ \lvert d_\phi \rvert, \lvert d_\psi \rvert\}$
    and
    $d_2 = \frac{d_\phi d_\psi}{d_1}$.}
  $$
\end{ex}

This simple example (or at least special cases of it) are known.  The
aim of the current section is to generalise this example to finitely generated torsion-free nilpotent groups.
Let $G$ be a finitely generated group and $\phi, \psi : G\rightarrow G$
two endomorphisms. 

\begin{lem}\label{red}
Let  $\phi, \psi : G\rightarrow G$ are two automorphisms.
Two elements $x,y$ of $G$ are $\psi^{-1}\phi$-conjugate if and only if
elements $\psi(x)$ and $\psi(y)$ are $(\psi,\phi)$-conjugate. Therefore the Reidemeister number
$R(\psi^{-1}\phi)$ is equal to $R(\phi,\psi)$.
For a tame pair of  commuting automorphisms $\phi, \psi : G\rightarrow G$
 the coicidence Reidemeister zeta function $R_{\phi,\psi}(z)$ is equal to the Reidemeister zeta  function $ R_{\psi^{-1}\phi}(z)$.
\end{lem}

\begin{proof} If $x$ and $y$ are $\psi^{-1}\phi$-conjugate, then there is
a  $g \in G$ such that $x=g y \psi^{-1}\phi(g^{-1})$.
This implies $\psi(x)=\psi(g)\psi(y)\phi(g^{-1})$. So
 $\psi(x)$ and  $\psi(y)$ are  $(\phi,\psi)$-conjugate.  The converse statement
follows if we move in opposite direction in previous implications. 
\end{proof}
We assume  $X$ to be a connected, compact polyhedron and $f:X\rightarrow X$ to be a continuous map.
The Lefschetz zeta function of a discrete dynamical system $f^n$ is defined as
$
L_f(z) := \exp\left(\sum_{n=1}^\infty \frac{L(f^n)}{n} z^n \right),
$
where
\begin{equation*}\label{Lef}
 L(f^n) := \sum_{k=0}^{\dim X} (-1)^k \tr\Big[f_{*k}^n:H_k(X;Q)\to H_k(X;Q)\Big]
\end{equation*}
is the Lefschetz number of the iterate $f^n$ of $f$. The Lefschetz zeta function is a rational function of $z$ and is given by the formula:
$$
L_f(z) = \prod_{k=0}^{\dim X}
          \det\big({I}-f_{*k}.z\big)^{(-1)^{k+1}}.
$$

In this section we consider finitely generated torsion-free nilpotent group $\Gamma$. It is well known \cite{mal} that such group $\Gamma$ is a uniform discrete subgroup of a simply connected nilpotent Lie group $G$ (uniform means that the coset space $G/ \Gamma$ is compact). The coset space $M=G/ \Gamma$ is called a nilmanifold.
 Since $\Gamma=\pi_1(M)$ and $M$  is a $K(\Gamma,
1)$, every endomorphism $\phi:\Gamma \to \Gamma $ can be realized by a selfmap
$f:M\to M$ such that $f_*=\phi$ and thus $R(f)=R(\phi)$. Any endomorphism  $\phi:\Gamma \to \Gamma $ can be uniquely extended to an endomorphism  $F: G\to G$. Let $\tilde F:\tilde G\to \tilde G $ be the corresponding Lie algebra endomorphism induced from $F$. 
 
\begin{lem}(cf. Theorem 23 of \cite{Fel00} and Theorem 5 of \cite{fhw})\label{nilpot}
Let $\phi : \Gamma\rightarrow \Gamma$ be a tame endomorphism of
 a finitely generated torsion free nilpotent group.
Then the Reidemeister  zeta function $R_\phi(z)=R_f(z)$ is a rational function and is equal to
\begin{equation}
R_\phi(z)=R_f(z)=L_f((-1)^pz)^{(-1)^r},
\end{equation}
 where  $p$ the number of $\mu\in Spectr(\tilde F)$ such that
$\mu <-1$, and $r$ the number of real eigenvalues of $\tilde F$ whose
absolute value is $>1$.
\end{lem}

Every pair  of automorphisms $\phi, \psi : \Gamma\rightarrow \Gamma$ of  finitely generated torsion free nilpotent group $\Gamma$ can be realized by a pair of homeomorphisms
$f,g :M\to M$ such that $f_*=\phi$ , $g_*=\psi$ and thus $R(g^{-1}f)=R(\psi^{-1}\phi)=R(\phi,\psi)$. An automorphism  $\psi^{-1}\phi:\Gamma \to \Gamma $ can be uniquely extended to an automorphism  $L: G\to G$. Let $\tilde L:\tilde G\to \tilde G $ be the corresponding Lie algebra automorphism induced from $L$. 

Lemma \ref{red} and Lemma \ref{nilpot} imply the following

\begin{teo}\label{nil2}
Let  $\phi, \psi : \Gamma\rightarrow \Gamma$ be a tame pair  of commuting automorphisms of a finitely generated torsion-free nilpotent group $\Gamma$.
Then the coincidence  Reidemeister  zeta function $R_{\phi,\psi}(z)$  is a rational function and is equal to
\begin{equation}
R_{\phi,\psi}(z)=R_{\psi^{-1}\phi}(z)= R_{g^{-1}f}(z) = L_{g^{-1}f}((-1)^pz)^{(-1)^r},
\end{equation}
 where  $p$ the number of $\mu\in Spectr(\tilde L)$ such that
$\mu <-1$, and $r$ the number of real eigenvalues of $\tilde L$ whose
absolute value is $>1$.
\end{teo}

 For an arbitrary group $G$, we can define the $k$-fold commutator group $\gamma_k(G)$
inductively as
$\gamma_1(G) := G$
and $\gamma_{k+1}(G) := [G, \gamma_k(G)]$.
Let $G$ be a group. For a subgroup $H \leqslant G$, we define the isolator $\sqrt[G]{H}$ of $H$ in $G$ as:
$\sqrt[G]{H} = \{g \in G \mbox{ } | \mbox{ } g^n \in H \mbox{ for some } n \in \mathbb{N} \}.$
Note that the isolator of a subgroup $H \leqslant G$ doesn't have to be a subgroup in general. For example, the isolator of the trivial group is the set of torsion elements of G.
\begin{lem}
(see \cite{dek}, Lemma 1.1.2 and Lemma 1.1.4)
Let $G$ be a group.
 Then
\begin{enumerate}
\item[(i)] for all $k\in\mathbb{N}, \sqrt[G]{\gamma_k(G)}$ is a fully characteristic subgroup of $G$,
\item[(ii)] for all $k\in\mathbb{N}$, the factor $G/\sqrt[G]{\gamma_k(G)}$ is torsion-free,
\item[(iii)] for all $k,l\in\mathbb{N}$,  the commutator $[\sqrt[G]{\gamma_k(G)},\sqrt[G]{\gamma_l(G)}]\leq \sqrt[G]{\gamma_{k+l}(G)}$,
\item[(iv)] for all $k,l\in\mathbb{N}$ such that $k\geqslant l \mbox{ if } M:=\sqrt[G]{\gamma_l(G)}$, then $$\sqrt[G/M]{\gamma_k(G/M)} = \sqrt[G]{\gamma_l(G)}/M. $$
\end{enumerate}
\end{lem}

We define the adapted lower central series of a group $G$ as 
$$ G = \sqrt[G]{\gamma_1(G)}\geqslant \sqrt[G]{\gamma_2(G)} \geqslant ... \sqrt[G]{\gamma_k(G)} \geqslant ... ,$$\\
where $\gamma_k(G)$ is the $k$-th commutator of $G$.

The adapted lower central series will terminate if and only if $G$ is a torsion-free, nilpotent group. Moreover, all factors $\sqrt[G]{\gamma_k(G)}/\sqrt[G]{\gamma_{k+1}(G)}$ are torsion-free.

We are particularly interested in the case where $G$ is a finitely generated, torsion-free, nilpotent group. In this case
the factors of the adapted lower central series are finitely generated,
torsion-free, abelian groups, i.e. for all $k \in \mathbb{N}$
 we have that
 
$\sqrt[G]{\gamma_k(G)}/\sqrt[G]{\gamma_{k+1}(G)}\cong \mathbb{Z}^{d_ k},$ for some $d_ k \in \mathbb{N}. $

Let $N$ be a normal subgroup of a group $G$ and $\phi,\psi \in End(G)$ with $\phi(N)\subseteq N, \psi(N)\subseteq N$.
We denote the restriction of $\phi$ to $N$ by $\phi|_N$,  $\psi$ to $N$ by $\psi|_N$ and the induced endomorphisms on the quotient $G/N$ by $\phi',\psi'$ respectively. We then get the following commutative diagrams with exact rows:
$ $\\
\begin{equation}
\commdiag{N}{\phi|_N\mbox{,}\psi|_N}{G}{\phi\mbox{,}\psi}{G/N}{\phi'\mbox{,}\psi'}
\end{equation}

Note that both $i$ and $p$ induce functions $\hat{i}, \hat{p}$ on the set of Reidemeister classes so that the sequence\\
\begin{center}
\begin{tikzcd}
\rcoincl{|_N} \arrow[r,"\hat{i}"] & \rcoincl{} \arrow[r,"\hat{p}"] & \rcoincl{'} \arrow[r] & 0
\end{tikzcd}
\end{center}
is exact, i.e. $\hat p$ is surjective and $\hat p ^{-1}[1] = im(\hat i)$, where 1 is the identity element of $G/N$ (see also \cite{dg}).
\begin{lem}\label{ineq}
If $R(\phi|_N,\psi|_N)<\infty, R(\phi',\psi'))<\infty$ 
and $N\subseteq Z(G)$, 
then\\ $R(\phi,\psi)\leqslant R(\phi|_N,\psi|_N)R(\phi',\psi')$.
\end{lem}
\begin{proof} Let
$\reidclasset{\phi|_N}{\psi|_N}{n}{}$ and \\
$\reidclasset{\phi'}{\psi'}{g}{N}$
be the $(\phi|_N,\psi|_N)$ Reidemeister classes and $(\phi',\psi')$ Reidemeister classes respectively.
Let us take $g\in G$. Then $gN\in [g_iN]_{{\phi'},{\psi'}}$ for some $i$, so there exists $hN\in G/N$ such that
$$gN = \psi'(hN)g_iN\phi'(hN)^{-1} = \psi'(h)g_i\phi'(h)^{-1}N.$$
It follows that there exists $n\in N$ such that
$g = \psi'(h)g_i\phi'(h)^{-1}n$.
In turn $n\in [n_j]_{\phi|_N,\psi|_N}$ for some $j$, hence there exists
 $m\in N$ such that
$n = \psi|_N(m)n_j\phi|_N(m)^{-1}$. 
Since $n,m,n_j,\psi|_N(m), \phi|_N(m)\in N\subset Z(G)$, 
it follows that
$$ g = [\psi(hm)](g_in_j)[\phi(hm)^{-1}], $$
i.e $g\in [g_in_j]_{\phi,\psi}$. Since this is true for arbitrary $g\in G$ we obtain that $$R(\phi,\psi)\leqslant R(\phi|_N,\psi|_N)R(\phi',\psi').$$

\end{proof}

\begin{teo}\label{rationality_of_rzf}

Let $N$ be a finitely generated, torsion-free, nilpotent group and
$$\adaptedlcs$$
be an adapted lower central series of $N$. Suppose that 
$ R(\phi,\psi)< \infty $ and
$\rnum{k}<\infty$ for a pair $\phi,\psi$ of endomorphisms of $N$ and for every
pair $\phi_k,\psi_k$, of induced  endomorphisms on the finitely generated torsion-free abelian factors 
$$
\sqrt[N]{\gamma_k(N)}/\sqrt[N]{\gamma_{k+1}(N)}\cong \mathbb{Z}^{d_ k},~ d_ k \in \mathbb{N} ,~ 1\leq k\leq c,
$$
then
$$
\rnum{}=\prod_{k=1}^{c}\rnum{k}.
$$ 

\end{teo}
\begin{proof}
 We will prove the product 
formula for the coincidence Reidemeister numbers by induction on the length of an adapted lower central series. Let us denote $\isol k$ as $N_k$. If $c = 1$,
the result follows trivially. Let $c > 1$ and assume the product formula holds for a central series of length $c - 1$. Let $\phi,\psi \in End(N)$, then $\phi(N_c)\subseteq N_c, \psi(N_c)\subseteq N_c$ and hence we have
the following commutative diagram of short exact sequences: \\

\begin{center}
\commdiag{N_c}{\phi_c\mbox{,}\psi_c}{N}{\phi\mbox{,}\psi}{N/N_c}{\phi'\mbox{,}\psi'},
\end{center}
$ $\\
where $\phi_c,\psi_c$ are induced endomorphisms on the $N_c$.
The quotient $\fctgp{}{c}$ is a finitely generated, nilpotent group with a central series
$$\centralquot$$
of length $c-1$.

Every factor of this series is of the form 
$$(\fctgp{k}{c})/(\fctgp{k+1}{c})\cong \fctgp{k}{k+1}$$
 by the third isomorphism theorem, hence it is
also torsion-free. Moreover, because of this natural isomorphism we know that
for every induced pair of endomorphisms $(\phi'_k,\psi'_k)$ on $(\fctgp{k}{c})/(\fctgp{k+1}{c})$ it is true that
$\rprim{k}=\rnum{k}.$

The assumptions of the theorem imply, that $\rprim{}<\infty$ and 
that $\rnum{c}<\infty$.

Moreover,    let $\primcl{1},...,\primcl{n}$
 be the $(\phi',\psi')$ -- Reidemeister classes and 
 $\cclass{1},...,\cclass{m}$ - the $(\phi'_c,\psi'_c)$-Reidemeister classes. Since $M_c\subseteq Z(N)$, by Lemma \ref{ineq} we obtain that
$\rnum{}\leqslant \rnum{c}\rprim{}.$ 

To prove the opposite inequality it suffices to prove that every Reidemeister class $\rclass$ represents a different
$(\phi,\psi)$-Reidemeister class. Then we obtain
$$\rnum{} = \rnum{c}\rprim{}$$
and then the theorem follows from the induction hypothesis.

Suppose, that there exists some $h\in N$ such that
$ c_ig_j = \psi(h)c_ag_b\phi(h)^{-1}.$

Then by taking the projection to $\fctgp{}{c}$ we find that
$$ g_jN_c = p(c_ig_j) = p(\psi(h)c_ag_b\phi(h)^{-1}) = \psi'(hN_c)(g_bN_c)\phi'(hN_c)^{-1}.$$

Hence $\primcl{j}=\primcl{b}$. Assume that 
$c_ig_j = \psi(h)c_ag_j\phi(h)^{-1}.$
If $h\in\ N_c\subseteq Z(N)$, then 
$c_ig_j = \psi(h)c_a\phi(h)^{-1}g_j$
and consequently $\cclass{i} = \cclass{a}$, so let us assume that $h \notin N_c$ and that $N_k$ is the smallest group in the central series which contains $h$. Then
$$c_ig_j = \psi(h)c_ag_j\phi(h)^{-1} \Leftrightarrow$$
$g_jc_i = \psi(h)c_ag_j\phi(h)^{-1} \Leftrightarrow$
$$c_i = g_j^{-1}\psi(h)c_ag_j\phi(h)^{-1} \Leftrightarrow$$
$c_i = g_j^{-1}\psi(h)g_j\phi(h)^{-1}c_a \Leftrightarrow$
$$c_ic_a^{-1} = g_j^{-1}\psi(h)g_j\phi(h)^{-1}$$
and therefore
\begin{equation*}
\begin{split}
c_ic_a^{-1}N_{k+1}  & = g_j^{-1}\psi(h)g_j\phi(h)^{-1}N_{k+1} \\
& = [g_j,\psi(h)^{-1}](\psi(h)\phi(h)^{-1})N_{k+1}
\end{split}
\end{equation*}
As $c_ic_a^{-1}\in N_c \subseteq N_{k+1}$ and  $[g_j,\psi(h)^{-1}]\in N_{k+1}$ , we find that
$$(\phi_k)(hN_{k+1}) = (\psi_k)(hN_{k+1}).$$
That means that the set of coincidence points Coin$(\phi'_k,\psi'_k)\neq \{1\}$, which implies that $\rprim{}{}=\infty$ and this contradicts assumption.
\end{proof}

\begin{teo}
 Let
  $\phi, \psi \colon N \to N$ be a tame pair of endomorphisms of a
 finitely generated  torsion-free nilpotent group~$N$.  Let $c$
  denote the nilpotency class of~$N$ and, for $1 \le k \le c$, let
  $\phi_k ,\psi_k \colon G_k \to G_k$, $1\leq k\leq c$, denote the tame pairs of  induced
  endomorphisms of the finitely generated torsion-free abelian factor groups
  $$G_k=N_k/N_{k+1}=\sqrt[N]{\gamma_k(N)}/\sqrt[N]{\gamma_{k+1}(N)}\cong \mathbb{Z}^{d_ k},$$ for some $d_ k \in \mathbb{N} $,  that arise from an adapted lower central
  series of~$N$.  Then the following hold.
  
 \textup{(1)}  For each $n \in \mathbb{N}$,
  \[
    R(\phi^n, \psi^n) = \prod_{k=1}^c R(\phi_k^{\, n},
    \psi_k^{\, n}) \qquad \text{for $n \in \mathbb{N}$.}
  \]
  
   \textup{(2)} For $1 \le k \le c$, let
  \[
    \phi_{k,\mathbb{Q}}, \psi_{k,\mathbb{Q}} \colon
    G_{k,\mathbb{Q}} \to
    G_{k,\mathbb{Q}}
  \]
  denote the extensions of $\phi_k, \psi_k$ to the divisible hull
  $G_{k,\mathbb{Q}} = \mathbb{Q} \otimes_\mathbb{Z} G_k \cong
  \mathbb{Q}^{d_k}$ of~$G_k$.  Suppose that each pair of endomorphisms
  $\phi_{k,\mathbb{Q}}, \psi_{k,\mathbb{Q}}$ is simultaneously
  triangularisable.  Let $\xi_{k,1}, \ldots, \xi_{k,d_k}$ and
  $\eta_{k,1}, \ldots, \eta_{k,d_k}$ be the eigenvalues of
  $\phi_{k,\mathbb{Q}}$ and $\psi_{k,\mathbb{Q}}$ in  the field~$\mathbb{C}$, including
  multiplicities, ordered so that, for $n \in \mathbb{N}$, the
  eigenvalues of
  $\phi_{k,\mathbb{Q}}^{\, n} - \psi_{k,\mathbb{Q}}^{\, n}$ are
  $\xi_{k,1}^{\, n} - \eta_{k,1}^{\, n}, \ldots, \xi_{k,d_k}^{\, n} -
  \eta_{k,d_k}^{\,n}$. 
Then   for each $n \in \mathbb{N}$,
  \begin{equation} \label{equ:key-formula-2}
    R(\phi_k^{\, n},\psi_k^{\, n}) = \prod_{i=1}^{d_k }\lvert \xi_{k,i}^{\, n} -
    \eta_{k,i}^{\, n} \rvert ;
    \end{equation}
     \smallskip
  \textup{(3)}
  Moreover, suppose that $\lvert \xi_{k,i} \rvert \neq \lvert \eta_{k,i}
  \rvert$ for $1 \le k \le c$, $1 \le i \le d_k$.
If $\phi,\psi$ is a tame pair of endomorphisms of $N$ and $\phi_{k,\mathbb{Q}}, \psi_{k,\mathbb{Q}}$, $1 \le k \le c$, are simultaneously
  triangularisable  pairs of endomorphisms of $G_{k,\mathbb{Q}}$,
then the coincidence Reidemeister zeta function   $R_{\phi,\psi}(z)$ is a rational function.
\end{teo}

\begin{proof}

The coincidence Reidemeister number $R(\phi, \psi)$ of  automorphisms $\phi,\psi $ of  an Abelian group $G$  coincides with the cardinality of the  quotient group $ \coker(\phi-\psi)=G/{\rm Im}(\phi-\psi)$
(or $\coker(\psi -\phi)=G/{\rm Im}(\psi-\phi)$). 

For $1 \le k \le c$, let tame pairs $\phi_k ,\psi_k \colon G_k \to G_k$ of  induced endomorphisms of the finitely generated torsion-free abelian factor groups $G_k=N_k/N_{k+1}\cong \mathbb{Z}^{d_ k},$  are represented by integer matrices $A_k,B_k\in M_{d_k}(\mathbb{Z})$ associated to them respectively. There is a diagonal integer matrix
$ C_k = \diag(c_1,. . . , c_{d_ k})$ such that $C_k = M_k(A_k - B_k)N_k$; where $M_k$ and $N_k$
are unimodular matrices. Now we have
$ \det C_k= \det(A_k - B_k)$  and the order of
the cokernel of $\phi_k -\psi_k $ is the order of the group $\mathbb{Z}/c_1\mathbb{Z}\oplus  . . . . .     \oplus \mathbb{Z}/c_{d_k}\mathbb{Z} .$
Thus the order of the cokernel of $ \phi_k -\psi_k $
is $|\coker(\phi_k -\psi_k )|= |c_1 \cdot \cdot \cdot  c_{d_ k}|= |\det C_k| = |\det(\phi_k -\psi_k )| .$

Then   for each $n \in \mathbb{N}$ and $1 \le k \le c$,
 $
    R(\phi_k^{\, n},\psi_k^{\, n}) =|\coker(\phi_k -\psi_k )|= |\det(\phi_k -\psi_k )|=|\det(\phi_{k,\mathbb{Q}}- \psi_{k,\mathbb{Q}})| =\prod_{i=1}^{d_k }\lvert \xi_{k,i}^{\, n} -
    \eta_{k,i}^{\, n} \rvert .
  $

Now we will prove the rationality of $R_{\phi,\psi}(z)$.
We open up the absolute values in the product
  $R(\phi_k^{\, n},\psi_k^{\, n}) = \prod_{i=1}^{d_k }\lvert \xi_{k,i}^{\, n} -
    \eta_{k,i}^{\, n} \rvert ,$ $1 \le k \le c$. Complex eigenvalues $\xi_{k,i}$ in the spectrum of
  $\phi_{k,\mathbb{Q}}$, respectively $\eta_{k,i}$ in the spectrum of
  $\psi_{k,\mathbb{Q}}$, appear in pairs with their complex conjugate
  $\overline{\xi_{k,i}}$, respectively $\overline{\eta_{k,i}}$.

  Moreover, such pairs can be lined up with one another in a
  simultaneous triangularisation as follows. Write
    $\phi_{k,\mathbb{C}}, \psi_{k,\mathbb{C}}$ for the induced endomorphisms of the
    $\mathbb{C}$-vector space
    $V = \mathbb{C} \otimes_\mathbb{Q} G \cong \mathbb{C}^{d_k} $.  If $v \in V$ is, at the same time, an eigenvector
    of $\phi_{k,\mathbb{C}}$ with complex eigenvalue $\xi_{k,d_k}$ and an eigenvector
    of $\psi_{k,\mathbb{C}}$ with eigenvalue $\eta_{k,d_k}$, then there is $w \in V$ such
    that $w$ is, at the same time, an eigenvector of $\phi_{k,\mathbb{C}}$ with
    eigenvalue $\overline{\xi_{k,d_k}} \ne \xi_{k,d_k}$ and an eigenvector of
    $\psi_{k,\mathbb{C}}$ with eigenvalue $\overline{\eta_{k,d_k}}$, possibly equal
    to~$\eta_{k,d_k}$.  Thus we can start our complete flag of
    $\{\phi,\psi\}$-invariant subspaces of $V$ with
    $\{0\} \subset \langle v \rangle \subset \langle v,w \rangle$, and
    proceed with $V/\langle v,w \rangle$ by induction to produce the
    rest of the flag in the same way, treating complex eigenvalues of
    $\psi_{k,\mathbb{C}}$ in the same way as they appear.
     If at least one of $\xi_{k,i}, \eta_{k,i}$ is complex so
    that these eigenvalues of $\phi_\mathbb{Q}$ and
    $\psi_\mathbb{Q}$ are paired with eigenvalues
    $\xi_{k,j} = \overline{\xi_{k,i}}, \eta_{k,j} = \overline{\eta_{k,i}}$, for
    suitable $j \ne i$, as discussed above, we see that
    \[
      \big\lvert \xi_{k,i}^{\, n} - \eta_{k,i}^{\, n} \big\rvert \;
      \big\lvert \xi_{k,j}^{\, n} - \eta_{k,j}^{\, n} \big\rvert =
      \big\lvert \xi_{k,i}^{\, n} - \eta_{k,i}^{\, n} \big\rvert ^{\,
        2} = \big( \xi_{k,i}^{\, n} - \eta_{k,i}^{\, n} \big) \cdot
      (\overline{\xi_{k,i}}^{\, n} - \overline{\eta_{k,i}}^{\, n} \big).
    \]
    If $\xi_{k,i}$ and $ \eta_{k,i}$ are both real eigenvalues of
    $\phi_\mathbb{Q}$ and $\psi_\mathbb{Q}$, not paired up with
    another pair of eigenvalues, then   exactly as in 
    Example \ref{ex1} above we have
    $\lvert \xi_{k,i}^{\, n} - \eta_{k,i}^{\, n} \rvert=
    \delta_{1,k,i}^{\, n} - \delta_{2,k,i}^{\, n}$,
  where  
    $\delta_{1,k,i} = \max\{\lvert \xi_{k,i} \rvert,\lvert \eta_{k,i}
    \rvert\} $ and
    $\delta_{2,k,i}=\frac{\xi_{k,i}\cdot\eta_{k,i}}{\delta_{1,k,i}}$.
  Hence we can expand each product $ R(\phi_k^{\, n},\psi_k^{\, n})$, $1 \le k \le c$ using an appropriate
  symmetric polynomial, to obtain for the Reidemeister numbers
 $R(\phi^n,\psi^n)$ an expression of the form
  \begin{equation}\label{dominant1}
      R(\phi^n, \psi^n) = \prod_{k=1}^c R(\phi_k^{\, n},
    \psi_k^{\, n}) =\prod_{k=1}^c \prod_{i=1}^{d_k }\lvert \xi_{k,i}^{\, n} -
    \eta_{k,i}^{\, n} \rvert = \sum_{j \in J} c_jw_j^{\, n},
  \end{equation}
  where~$J$ is a finite index set, $c_j \in \{-1,1\}$ and
  $\{ w_j \mid j \in J \} \subseteq \mathbb{C} \smallsetminus \{0\}$.
  Consequently, the coincidence Reidemeister zeta function can be
  written as
  \[
    R_{\phi,\psi}(z) =\exp\left(\sum_{n=1}^\infty \frac{R(\phi^n,\psi^n)}{n}z^n\right)= \exp \left( \sum_{j \in J} c_j \sum_{n=
        1}^{\infty} \frac{ (w_jz)^n} {n} \right).
  \]
  and it follows immediately that
  $R_{\phi,\psi}(z) = \prod_{j \in J} (1 - w_jz)^{-c_j}$ is a
  rational function.

\end{proof}


\begin{thebibliography}{99}


\bibitem{BMW} 
J.~Bell, R.~Miles, T.~Ward, Towards a P{\'o}lya--Carlson dichotomy
for algebraic dynamics, \emph{Indag. Math.(N.S.} \textbf{25} (2014), no.~4, 652-668.

\bibitem{ByCo18} J.~Byszewski and G.~Cornelissen, Dynamics on
    abelian varieties in positive characteristic, with an appendix by
  R.~Royals and T.~Ward,  Algebra Number Theory \textbf{12} (2018), 
  2185--2235.

\bibitem{Car}
F.~Carlson,  `\"{U}ber ganzwertige {F}unktionen', \emph{Math. Z.} \textbf{11}
  (1921), no.~1-2, 1--23.


\bibitem{CEW}
V.~Chothi, G.~Everest, and T.~Ward,  {$S$}-integer dynamical systems: periodic
  points, \emph{J. Reine Angew. Math.} \textbf{489} (1997), 99--132.

\bibitem{dek}
	Dekimpe K.
	Almost-Bieberbach groups: affine and polynomial structures.
	Vol. 1639. Lecture Notes in Mathematics. Springer-Verlag,
Berlin, 1996, pp. x+259

\bibitem{DeDu}
Dekimpe, K. and Dugardein, G.-J.
 Nielsen zeta functions for maps on infra-nilmanifolds are
  rational,
J. Fixed Point Theory Appl., {\bf 17}(2)(2015),  355--370.

\bibitem{DekTerBus}
 Karel Dekimpe, Sam Tertooy, and Iris Van~den Bussche,
 Reidemeister zeta functions of low-dimensional
  almost-crystallographic groups are rational,
 Communications in Algebra, {\bf 46} (9)(2018), 4090--4103.





\bibitem{Eisenbud}
D.~Eisenbud, Commutative algebra, volume 150 of Graduate Texts in Mathematics,
Springer - Verlag, New York, 1995, With a view toward algebraic geometry.

\bibitem{ESW}
G.~Everest, V.~Stangoe, and T.~Ward,  `Orbit counting with an isometric
  direction', in Algebraic and topological dynamics, in Contemp.
  Math. \textbf{385} (2005), pp.~293--302, Amer. Math. Soc., Providence, RI.

\bibitem{EPSW}
G.~Everest, A.~van~der Poorten, I.~Shparlinski, and T.~Ward, Recurrence
  sequences, in \emph{Mathematical Surveys and Monographs} \textbf{104}, Amer. Math. Soc., Providence, RI, 2003.

\bibitem{Fel91}
A.~L.~Fel'shtyn,
The Reidemeister zeta function and the computation of the Nielsen zeta function,
Colloq. Math., {\bf 62}, (1991), 153--166.

\bibitem{Fel00}
A.~Fel'shtyn,
Dynamical zeta functions, Nielsen theory and Reidemeister torsion,
Mem. Amer. Math. Soc., {\bf 699}, Amer. Math. Soc., Providence, R.I. 2000.
\bibitem{fh}
A. L. Fel'shtyn and R. Hill,
 The Reidemeister zeta function with applications to Nielsen theory and a connection with Reidemeister torsion,
 K-theory, {\bf 8} (1994), 367--393.
\bibitem{fhw}
A. L. Fel'shtyn, R. Hill and P. Wong,
Reidemeister numbers of equivariant maps,
Topology Appl., {\bf 67} (1995), 119--131
\bibitem{FK}
Alexander Fel'shtyn and Benjamin Klopsch,
 P\'olya--{C}arlson dichotomy for coincidence {R}eidemeister zeta
  functions via profinite completions.
 e-print, 2021,
arXiv:2102.10900(to appear in Indagationes Mathematicae, 2022).

\bibitem{FelLee}
 Alexander Fel'shtyn and Jong~Bum Lee, 
The Nielsen and Reidemeister numbers of maps on
  infra-solvmanifolds of type ({R}).
Topology Appl. {\bf 181}(2015), 62--103.

\bibitem{FeTr21}
 A.~Fel'shtyn, E.~Troitsky.
 P\'olya--Carlson dichotomy for dynamical zeta functions
and twisted Burnside-Frobenius theorem.
 Russ. J. Math. Phys. {\bf 28}(2021),  No.~4, 455--463.

\bibitem{FeTrZi20}
 A.~Fel'shtyn, E.~Troitsky, and M.~Zietek.
 New {Z}eta {F}unctions of {R}eidemeister {T}ype and the {T}wisted
  {B}urnside-{F}robenius {T}heory.
 Russ. J. Math. Phys. {\bf 27}(2020),  No.~2, 199--211.

\bibitem{FeZi20} A.~Fel'shtyn and M.~Zietek, Dynamical zeta
    functions of Reidemeister type and representations spaces,
  57--81, in: Contemp.\ Math.\ {\bf 744}(2020), 57--81, Amer.\ Math.\ Soc.,
  Providence, R.I.

\bibitem{FZ2} A.~Fel'shtyn and M.~Zietek, Dynamical zeta
    functions of Reidemeister type,
  Topological Methods Nonlinear Anal. {\bf 56}(2020), 433-455.



\bibitem{dg}
  Daciberg L. Gon\c calves, Peter N.-S. Wong: 
  Homogenous spaces in coincidence theory.
  Forum Math. {\bf17} (2005), 297--313


\bibitem{j}
B. Jiang,
Nielsen Fixed Point Theory,
 Contemp. Math. {\bf14}, Birkh\"auser, 1983.
 

\bibitem{Li}
L.~ Li,  On the rationality of the Nielsen zeta function, Adv. in Math. (China), {\bf23} (1994)
no.~3, 251--256.

\bibitem{LW}
D.~A. Lind and T.~Ward,  Automorphisms of solenoids and {$p$}-adic entropy,
  Ergodic Theory Dynam. Systems \textbf{8} (1988), no.~3, 411--419.
\bibitem{mal}
 A.~Mal'cev, On a class of homogeneous spaces.  Izvestiya
Akademii Nauk SSSR. Seriya Matemati\v ceskaya, \textbf{13 } (1949), 9-32.


\bibitem{Matsumura}
H.~Matsumura, Commutative ring theory, volume 8 of Cambridge Studies in Advanced Mathematics.
Cambridge University Press, Cambridge, second edition, 1989. Translated from the  Japanese
by M. Reid.
\bibitem{Mi07}
R.~Miles,  Zeta functions for elements of entropy rank-one actions,
  Ergodic Theory Dynam. Systems \textbf{27} (2007), no.~2, 567--582.

\bibitem{Mi}
R.~Miles,  Periodic points of endomorphisms on solenoids and related groups,
  Bull. Lond. Math. Soc. \textbf{40} (2008), no.~4, 696--704.
  
\bibitem{Mi13} R.~Miles, Synchronization points and
    associated dynamical invariants, Trans.\ Amer.\ Math.\ Soc.\
  \textbf{365} (2013), 5503--5524.

\bibitem{MW}
R.~Miles and T.~Ward,  The dynamical zeta function for commuting automorphisms of zero-dimensional groups,
  Ergodic Theory Dynam. Systems \textbf{38} (2018), no.~4, 1564-- 1587.

\bibitem{MyPoo}
G.~Myerson and A.~J. van~der Poorten,  Some problems concerning recurrence
  sequences, Amer. Math. Monthly \textbf{102} (1995), no.~8, 698--705.


\bibitem{Po}
G.~P{\'o}lya,  `\"{U}ber gewisse notwendige {D}eterminantenkriterien f\"ur die
  {F}ortsetzbarkeit einer {P}otenzreihe', \emph{Math. Ann.} \textbf{99} (1928),
  no.~1, 687--706.
\bibitem{Sch}
K.~Schmidt, Dynamical systems of algebraic origin, in Progress in
  Mathematics \textbf{128}, Birkh\"auser Verlag, Basel, 1995.
\bibitem{Segal}
S.~L. Segal.
 Nine introductions in complex analysis, volume \textbf{208} of 
  North-Holland Mathematics Studies.
Elsevier Science B.V., Amsterdam, revised edition, 2008.


\bibitem{Smale}
S. Smale,
Differentiable dynamical systems,
Bull. Amer. Math. Soc., \textbf {73} (1967), 747--817.

\bibitem{tr}
 Evgenij Troitsky,  Two examples related to the twisted {B}urnside-{F}robenius theory for
  infinitely generated groups,
Fundam. Appl. Math. \textbf {21}(2016),  No.~5, 231--239.

\bibitem{Weil}
A.~Weil, Basic number theory, in Die Grundlehren der
  mathematischen Wissenschaften, Band \textbf{144}, Springer-Verlag New York, Inc., New
  York, 1967.

\bibitem{Wong01}
P. Wong,
Reidemeister zeta function for group extensions,
J. Korean Math. Soc., \textbf {38} (2001), 1107--1116.




\end{thebibliography}
\end{document}